\documentclass[english]{article}
\usepackage[T1]{fontenc}
\usepackage[latin9]{inputenc}
\usepackage{color}
\usepackage{refstyle}
\usepackage{enumitem}
\usepackage{amsmath}
\usepackage{amsthm}
\usepackage{amssymb}
\usepackage{setspace}
\makeatletter

\AtBeginDocument{\providecommand\secref[1]{\ref{sec:#1}}}
\AtBeginDocument{\providecommand\lemref[1]{\ref{lem:#1}}}
\AtBeginDocument{\providecommand\thmref[1]{\ref{thm:#1}}}
\AtBeginDocument{\providecommand\enuref[1]{\ref{enu:#1}}}
\AtBeginDocument{\providecommand\exaref[1]{\ref{exa:#1}}}
\AtBeginDocument{}
\AtBeginDocument{\providecommand\corref[1]{\ref{cor:#1}}}
\RS@ifundefined{subref}
  {\def\RSsubtxt{section~}\newref{sub}{name = \RSsubtxt}}
  {}
\RS@ifundefined{thmref}
  {\def\RSthmtxt{theorem~}\newref{thm}{name = \RSthmtxt}}
  {}
\RS@ifundefined{lemref}
  {\def\RSlemtxt{lemma~}\newref{lem}{name = \RSlemtxt}}
  {}

\usepackage{enumitem}		
\theoremstyle{plain}
\newtheorem{thm}{\protect\theoremname}[section]
  \theoremstyle{definition}
  \newtheorem{defn}[thm]{\protect\definitionname}
  \theoremstyle{remark}
  \newtheorem{rem}[thm]{\protect\remarkname}
  \theoremstyle{plain}
  \newtheorem{lem}[thm]{\protect\lemmaname}
  \theoremstyle{plain}
  \newtheorem{prop}[thm]{\protect\propositionname}
 \newlist{casenv}{enumerate}{4}
 \setlist[casenv]{leftmargin=*,align=left,widest={iiii}}
 \setlist[casenv,1]{label={{\itshape\ \casename} \arabic*.},ref=\arabic*}
 \setlist[casenv,2]{label={{\itshape\ \casename} \roman*.},ref=\roman*}
 \setlist[casenv,3]{label={{\itshape\ \casename\ \alph*.}},ref=\alph*}
 \setlist[casenv,4]{label={{\itshape\ \casename} \arabic*.},ref=\arabic*}
  \theoremstyle{plain}
  \newtheorem{cor}[thm]{\protect\corollaryname}
  \theoremstyle{definition}
  \newtheorem{example}[thm]{\protect\examplename}

\@ifundefined{date}{}{\date{}}
\AtBeginDocument{\providecommand\condref[1]{\ref{cond:#1}}}\AtBeginDocument{\providecommand\defref[1]{\ref{def:#1}}}\AtBeginDocument{\providecommand\remref[1]{\ref{rem:#1}}}\AtBeginDocument{\providecommand\caseref[1]{\ref{case:#1}}}

\usepackage{tikz}
\usepackage{tocbibind}
\usepackage{youngtab}
\usepackage{authblk}
\usetikzlibrary{matrix}

\newcommand{\Rt}{\widetilde{\mathcal{R}}_E}
\newcommand{\Lt}{\widetilde{\mathcal{L}}_E}
\newcommand{\Ht}{\widetilde{\mathcal{H}}_E}

\newcommand{\Rs}{\mathcal{R}^{\ast}}
\newcommand{\Ls}{\mathcal{L}^{\ast}}

\DeclareMathOperator{\id}{id}

\DeclareMathOperator{\Rad}{Rad}
\DeclareMathOperator{\im}{\mathsf{im}}

\DeclareMathOperator{\PT}{\mathcal{PT}}

\DeclareMathOperator{\T}{\mathcal{T}}

\DeclareMathOperator{\Rc}{\mathcal{R}}
\DeclareMathOperator{\Lc}{\mathcal{L}}
\DeclareMathOperator{\Hc}{\mathcal{H}}
\DeclareMathOperator{\Dc}{\mathcal{D}}

\DeclareMathOperator{\dom}{\mathsf{dom}}

\DeclareMathOperator{\op}{op}
\DeclareMathOperator{\Reg}{Reg}
\DeclareMathOperator{\db}{\mathbf{d}}
\DeclareMathOperator{\rb}{\mathbf{r}}

\AtBeginDocument{}
\RS@ifundefined{factref}{\newref{fact}{name = Fact~,names = Facts~}}{}

\AtBeginDocument{\providecommand\caseref[1]{\ref{case:#1}}}
\RS@ifundefined{caseref}{\newref{case}{name = Case~,names = cases~}}{}

\AtBeginDocument{\providecommand\remref[1]{\ref{rem:#1}}}
\RS@ifundefined{remref}{\newref{rem}{name = Remark~,names = Renarks~}}{}

\AtBeginDocument{\providecommand\defref[1]{\ref{def:#1}}}
\RS@ifundefined{defref}{\newref{def}{name =Definition~,names = Definitions~}}{}

\AtBeginDocument{\providecommand\condref[1]{\ref{condition:#1}}}
\RS@ifundefined{condref}{\newref{cond}{name = Condition~,names = conditions~}}{}

\AtBeginDocument{\providecommand\exaref[1]{\ref{example:#1}}}
\RS@ifundefined{exaref}{\newref{exa}{name = Example~,names = examples~}}{}

\AtBeginDocument{}
\RS@ifundefined{propref}{\newref{prop}{name = Proposition~,names = propositions~}}{}

\AtBeginDocument{\providecommand\corref[1]{\ref{corollary:#1}}}
\RS@ifundefined{corref}{\newref{cor}{name = Corollary~,names = corollaries~}}{}

\def\RSlemtxt{Lemma~}
\def\RSthmtxt{Theorem~}
\def\RSsubtxt{Section~}

\usepackage{babel}
\providecommand{\corollaryname}{Corollary}
  \providecommand{\definitionname}{Definition}
  \providecommand{\examplename}{Example}
  \providecommand{\lemmaname}{Lemma}
  \providecommand{\propositionname}{Proposition}
  \providecommand{\remarkname}{Remark}
 \providecommand{\casename}{Case}
\providecommand{\theoremname}{Theorem}

\makeatother

\usepackage{babel}
  \providecommand{\corollaryname}{Corollary}
  \providecommand{\definitionname}{Definition}
  \providecommand{\examplename}{Example}
  \providecommand{\lemmaname}{Lemma}
  \providecommand{\propositionname}{Proposition}
  \providecommand{\remarkname}{Remark}
 \providecommand{\casename}{Case}
\providecommand{\theoremname}{Theorem}

\begin{document}

\title{Algebras of Ehresmann semigroups and categories}

\author{Itamar Stein \footnote{email: Steinita@gmail.com} \thanks{This paper is part of the author's PHD thesis, being carried out under
the supervision of Prof. Stuart Margolis. The author's research was
supported by Grant No. 2012080 from the United States-Israel Binational
Science Foundation (BSF) and by the Israeli Ministry of Science, Technology and Space.}}

\affil{ Department of Mathematics\\
 Bar Ilan University\\
 52900 Ramat Gan\\
 Israel}

\maketitle

\begin{abstract}
$E$-Ehresmann semigroups are a commonly studied generalization of
inverse semigroups. They are closely related to Ehresmann categories
in the same way that inverse semigroups are related to inductive groupoids.
We prove that under some finiteness condition, the semigroup algebra
of an $E$-Ehresmann semigroup is isomorphic to the category algebra
of the corresponding Ehresmann category. This generalizes a result
of Steinberg who proved this isomorphism for inverse semigroups and
inductive groupoids and a result of Guo and Chen who proved it for
ample semigroups. We also characterize $E$-Ehresmann semigroups whose
corresponding Ehresmann category is an EI-category and give some natural
examples.
\end{abstract}

\section{Introduction}

A semigroup $S$ is called inverse if every element $a\in S$ has
a unique inverse, that is, a unique $b\in S$ such that $aba=a$ and
$bab=b$. Inverse semigroups are fundamental in semigroup theory and
have many unique and important properties. For instance, they are
ordered with respect to a natural partial order and their idempotents
form a semilattice. Another important fact is their close relation
with inductive groupoids. More precisely, the Ehresmann-Schein-Nambooripad
theorem \cite[Theorem 8 of Section 4.1]{Lawson1998} states that the
category of all inverse semigroups is isomorphic to the category of
all inductive groupoids. For an extensive study of inverse semigroups
see \cite{Lawson1998}. There are several generalizations of inverse
semigroups that keep some of their good properties. In this paper
we discuss a generalization called $E$-Ehresmann semigroups. Let $E$ be a subsemilattice
of $S$. Define two equivalence relations $\Rt$ and $\Lt$ in the
following way. $a\Rt b$ if $a$ and $b$ have precisely the same
set of left identities from $E$ and likewise $a\Lt b$ if they have
the same set of right identities from $E$. Assume that every $\Rt$
and $\Lt$ class contains precisely one idempotent, denoted $a^{+}$
and $a^{\ast}$ respectively. $S$ is called $E$-Ehresmann if $\Rt$
is a left congruence and $\Lt$ is a right congruence, or equivalently (see \cite[Lemma 4.1]{Gould2010}),
if the two identities $(ab)^{+}=(ab^{+})^{+}$ and $(a^{\ast}b)^{\ast}=(ab)^{\ast}$
hold for every $a,b\in S$. If $S$ is regular and $E=E(S)$ is the
set of all idempotents of $S$ then being an $E$-Ehresmann semigroup
is equivalent to being an inverse semigroup. There is also a notion
of an Ehresmann category which is a generalization of an inductive
groupoid. The Ehresmann-Schein-Nambooripad theorem generalizes well
to $E$-Ehresmann semigroups and Ehresmann categories. Lawson proved
\cite{Lawson1991} that the category of all $E$-Ehresmann semigroups
is isomorphic to the category of all Ehresmann categories. In this
paper we discuss algebras of these objects over some commutative unital ring $\mathbb{K}$. In recent years, algebras
of semigroups related to Ehresmann semigroups have been studied by
a number of authors, see \cite{Guo2012,Guo2015,Ji2016}.

Steinberg \cite{Steinberg2006} proved that if $S$ is an inverse
semigroup where $E(S)$ is finite then its algebra is isomorphic to
the algebra of the corresponding inductive groupoid. Guo and Chen
\cite{Guo2012} generalized this isomorphism to the case of finite
ample semigroups. These are the $E$-Ehresmann semigroups such that
$E=E(S)$, every $\Rs$ and $\Ls$ class contains an idempotent (see
\cite{Gould2010} for the definition of the equivalence relations
$\Rs$ and $\Ls$) and the two ample conditions $ae=(ae)^{+}a$ and
$ea=a(ea)^{\ast}$ hold for every $a\in S$ and $e\in E(S)$. An important
example of an $E$-Ehresmann semigroup is the monoid $\PT_{n}$ of
all partial functions on an $n$-element set where $E$ is the semilattice
of all partial identity functions. The author proved \cite{Stein2016}
that the algebra of $\PT_{n}$ is isomorphic to the algebra of the
category of all surjections between subsets of an $n$-element set.
The category of surjections is in fact, the Ehresmann category associated to $\PT_n$ as an $E$-Ehresmann semigroup.
This result has led to some new results on the representation theory
of $\PT_{n}$. In this paper we generalize all these results and prove
that if the subsemilattice $E\subseteq S$ is principally finite (that
is, any principal down ideal is finite) then the semigroup algebra
$\mathbb{K}S$ (over any commutative unital ring $\mathbb{K}$) is
isomorphic to the category algebra $\mathbb{K}C$ of the corresponding
Ehresmann category. In \secref{Examples} we give some examples and
simple corollaries of this isomorphism. 

In order to apply this isomorphism to study the semigroup algebra
(and representations) of some $E$-Ehresmann semigroup, one need to
understand the algebra of the corresponding Ehresmann category. Hence
it is natural to consider Ehresmann categories whose algebras
are well understood to some extent. EI-categories, which are categories where every
endomorphism is an isomorphism, are such a family. If $\mathbb{K}$
is an algebraically closed field of good characteristic then there is a way to describe
the Jacobson radical of the algebra of a finite EI-category $C$ (\cite[Proposition 4.6]{Li2011}) and
its ordinary quiver (\cite[Theorem 4.7]{Li2011} or \cite[Theorem 6.13]{Margolis2012}).
In \secref{Ehresmann-EI-categories} we characterize $E$-Ehresmann
semigroups whose corresponding Ehresmann category is an EI-category.
We give two natural families of such semigroups: $(2,1,1)$-subalgebras
of $\PT_{n}$ and $E(S)$-Ehresmann semigroups. Let $S$ be a finite $E$-Ehresmann semigroup whose corresponding category $C$ is an EI-category and let $\mathbb{K}$ be a field such that the orders of the endomorphism groups of all objects in $C$ are invertible in $\mathbb{K}$. In \secref{Max-Semisimple-Image} we prove that in this case the maximal semisimple image of $\mathbb{K}S$  can be "seen" inside the semigroup $S$ itself. More precisely, we will prove that the inverse subsemigroup of $S$ that contains all elements corresponding to isomorphisms in $C$ spans an algebra which is isomorphic to the maximal semisimple image of $\mathbb{K}S$.

\section{Preliminaries }

\subsection{$E$-Ehresmann semigroups}

We denote as usual by $\Rc$, $\Lc$, $\Dc$ and $\Hc$ the Green's
relations on a semigroup. We assume that the reader is familiar with Green's
relations and other semigroup basics that can be found in \cite{Howie1995}.
Recall that a semilattice is a commutative semigroup of idempotents,
or equivalently, a poset such that any two elements have a meet. Let
$S$ be a semigroup. Denote by $E(S)$ its set of idempotents and
choose some $E\subseteq E(S)$ such that $E$ is a subsemilattice of $S$. We
define equivalence relations $\Rt$ and $\Lt$ on $S$ by 
\[
a\Rt b\iff(\forall e\in E\quad ea=a\Leftrightarrow eb=b)
\]
and 
\[
a\Lt b\iff(\forall e\in E\quad ae=a\Leftrightarrow be=b).
\]
We also define $\Ht=\Rt\cap\Lt$. It is easy to see that $\Rc\subseteq\Rt$,
$\Lc\subseteq\Lt$ and $\Hc\subseteq\Ht$.
\begin{defn}
\label{def:EhresmannSemigroup} A semigroup $S$ with a distinguished
semilattice $E\subseteq E(S)$ is called \emph{left $E$-Ehresmann
}if the following two conditions hold.\end{defn}
\begin{enumerate}
\item \label{cond:Ehresmann1} Every $\Rt$ class contains precisely one
idempotent from $E$. 
\item \label{cond:Ehresmann2}$\Rt$ is a left congruence. \end{enumerate}
\begin{rem}
It is easy to see that an $\Rt$ class cannot contain more than one
idempotent from $E$ so \condref{Ehresmann1} can be replaced by
the requirement that every $\Rt$ class contains at least one idempotent
from $E$.  
\end{rem}
In any semigroup that satisfies \condref{Ehresmann1} we denote
by $a^{+}$ the unique idempotent from $E$ in the $\Rt$ class of
$a$. Note that $a^{+}$ is the unique minimal element $e$ of the
semilattice $E$ that satisfies $ea=a$.

Note that if $S$ is a finite monoid, \condref{Ehresmann1}
is equivalent to the requirement that $1\in E$. Indeed, $1$ is the only left identity of itself so \condref{Ehresmann1} implies that $1\in E$. On the other hand, assume $1\in E$ and take $e$ to be the product of all idempotents of $E$ which are left identity for $a$. This product is not empty since $1\in E$. It is clear that $e\in E$ and $e \Rt a$ so \condref{Ehresmann1} holds.

 \condref{Ehresmann2} of  \defref{EhresmannSemigroup} has the following equivalent characterization (for proof see \cite[Lemma 4.1]{Gould2010}). 
\begin{lem}
\label{lem:LeftRighCongIdentities}Let $S$ be a semigroup with a
distinguished semilattice $E\subseteq E(S)$ such that \condref{Ehresmann1}
of \defref{EhresmannSemigroup} holds. Then
$\Rt$ is a left congruence if and only if $(ab)^{+}=(ab^{+})^{+}$
for every $a$,$b\in S$. 
\end{lem}
Dually, we can consider semigroups for which every $\Lt$ class contains
a unique idempotent. We denote the unique idempotent in the $\Lt$
class of $a$ by $a^{\ast}$. Such semigroup is called right $E$-Ehresmann
if $\Lt$ is a right congruence, or equivalently if $(ab)^{\ast}=(a^{\ast}b)^{\ast}$
for every $a,b\in S$. 
\begin{defn}
A semigroup $S$ with a distinguished semilattice $E\subseteq E(S)$
is called \emph{$E$-Ehresmann }if it is both left and right $E$-Ehresmann.
\end{defn}
Note that any inverse semigroup $S$ is an $E$-Ehresmann semigroup
when one choose $E=E(S)$. In this case $a^{+}=aa^{-1}$ and $a^{\ast}=a^{-1}a$.

As may be hinted by \lemref{LeftRighCongIdentities}, $E$-Ehresmann
semigroups form a variety of bi-unary semigroups, that is, semigroups with two additional binary operations. The proof of the
following proposition can be found in \cite[Lemma 2.2]{Gould2010b} and the discussion
following it. 
\begin{prop}
$E$-Ehresmann semigroups form precisely the variety of $(2,1,1)$-algebras
(where $^{+}$ and $^{\ast}$ are the unary operations) subject to
the identities: 
\begin{align*}
x^{+}x & =x,\,(x^{+}y^{+})^{+}=x^{+}y^{+},\,x^{+}y^{+}=y^{+}x^{+},\,x^{+}(xy)^{+}=(xy)^{+},\,(xy)^{+}=(xy^{+})^{+}\\
xx^{\ast} & =x,\,(x^{\ast}y^{\ast})^{\ast}=x^{\ast}y^{\ast},\:x^{\ast}y^{\ast}=y^{\ast}x^{\ast},\,(xy)^{\ast}y^{\ast}=(xy)^{\ast},\,(xy)^{\ast}=(x^{\ast}y)^{\ast}\\
x(yz) & =(xy)z,\,(x^{+})^{\ast}=x^{+},\,(x^{\ast})^{+}=x^{\ast}.
\end{align*}

\end{prop}
One of the advantages of the varietal point of view is that one does
not need to mention the set $E$ as it is the image of the unary operations:
\[
E=\{a^{\ast}\mid a\in S\}=\{a^{+}\mid a\in S\}.
\]

Let $S$ be an inverse semigroup. It is well known that $S$ affords
a natural partial order defined by $a\leq b$ if and only if $a=aa^{-1}b$,
or equivalently, $a=ba^{-1}a$. In the general case of $E$-Ehresmann
semigroups this partial order splits into right and left versions.
We say that $a\leq_{r}b$ if and only if $a=a^{+}b$. Dually, $a\leq_{l}b$
if and only if $a=ba^{\ast}$. 
\begin{prop}[{{\cite[Section 7]{Gould2010}}}]
\end{prop}
\begin{enumerate}
\item $\leq_{r}$ and $\leq_{l}$ are indeed partial orders on $S$. 
\item $a\leq_{r}b$ if and only if $a=eb$ for some $e\in E$. Dually, $a\leq_{l}b$
if and only if $a=be$ for some $e\in E$. 
\end{enumerate}

\subsection{Ehresmann Categories}

All categories in this paper will be \emph{small}, that is, their
morphisms form a set. Hence we can regard a category $C$, as a set
of objects, denoted $C^{0}$ and a set of morphisms, denoted $C^{1}$.
We will identify an object $e\in C^{0}$ with its identity morphism
$1_{e}$ so we can regard $C^{0}$ as a subset of $C^{1}$. We denote
the domain and range of a morphism $x\in C^{1}$ by $\db(x)$ and
$\rb(x)$ respectively. Recall that the multiplication $x\cdot y$
of two morphisms is defined if and only if $\rb(x)=\db(y)$. We also
denote the fact that $\rb(x)=\db(y)$ by $\exists x\cdot y$. Note
that in this paper we multiply morphisms (and functions) from left
to right. Recall that a \emph{groupoid} is a category where every
morphism is invertible.
\begin{defn}
A category $C$ equipped with a partial order $\leq$ on its morphisms
is called a \emph{category with order }if the following hold. 
\begin{enumerate}[label=(CO\arabic*)]
\item \label{enu:CatWithOrderLeqOfDomRan}If $x\leq y$ then $\db(x)\leq\db(y)$
and $\rb(x)\leq\rb(y)$. 
\item \label{enu:CatWithOrderCompatability}If $x\leq y$, $u\leq v$, $\exists x\cdot u$
and $\exists y\cdot v$ then $x\cdot u\leq y\cdot v$. 
\item If $x\leq y$, $\db(x)=\db(y)$ and $\rb(x)=\rb(y)$ then $x=y$. 
\end{enumerate}
\end{defn}

\begin{defn}
A category $C$ equipped with two partial orders on morphisms $\leq_{r}$,
$\leq_{l}$ is called an \emph{Ehresmann category} if the following
hold: 
\begin{enumerate}[label=(EC\arabic*)]
\item $C$ equipped with $\leq_{r}$ (respectively, $\leq_{l}$) is a category
with order. 
\item \label{enu:EhresmannCatrestriction}If $x\in C^{1}$ and $e\in C^{0}$
with $e\leq_{r}\db(x)$ then there exists a unique \emph{restriction}
$(e\mid x)\in C^{1}$ satisfying $\db((e\mid x))=e$ and $(e\mid x)\leq_{r}x$. 
\item If $x\in C^{1}$ and $e\in C^{0}$ with $e\leq_{l}\rb(x)$ then there
exists a unique \emph{co-restriction} $(x\mid e)\in C^{1}$ satisfying
$\rb((x\mid e))=e$ and $(x\mid e)\leq_{l}x$. 
\item For $e,f\in C^{0}$ we have $e\leq_{r}f$ if and only if $e\leq_{l}f$.\label{enu:PartialOrdersEqualOnObjects} 
\item $C^{0}$ is a semilattice with respect to $\leq_{r}$ (or $\leq_{l}$,
since they are equal on $C^{0}$ by \ref{enu:PartialOrdersEqualOnObjects}). 
\item $\leq_{r}\circ\leq_{l}=\leq_{l}\circ\leq_{r}$. 
\item \label{enu:EhresmannCatWedgeCorestriction}If $x\leq_{r}y$ and $f\in C^{0}$
then $(x\mid\rb(x)\wedge f)\leq_{r}(y\mid\rb(y)\wedge f)$. 
\item If $x\leq_{l}y$ and $f\in C^{0}$ then $(\db(x)\wedge f\mid x)\leq_{l}(\db(y)\wedge f\mid y)$. 
\end{enumerate}
\end{defn}
\begin{rem}
Note that for every morphism $x$ of an Ehresmann category we have
$(x\mid\rb(x))=x=(\db(x)\mid x)$. 
\end{rem}
From every $E$-Ehresmann semigroup $S$ we can construct an Ehresmann
category ${\bf C}(S)=C$ in the following way. The object set of ${\bf C}(S)$
is the set $E$ and morphisms of ${\bf C}(S)$ are in one-to-one correspondence
with elements of $S$. For every $a\in S$ we associate a morphism
$C(a)\in C^{1}$ such that $\db(C(a))=a^{+}$ and $\rb(C(a))=a^{\ast}$.
If $\exists C(a)\cdot C(b)$ then $C(a)\cdot C(b)=C(ab)$. Finally
$C(a)\leq_{r}C(b)$ ($C(a)\leq_{l}C(b)$) whenever $a\leq_{r}b$ (respectively,
$a\leq_{l}b$) according to the partial order of $S$ defined above. 

\begin{prop}[{{\cite[Proposition 4.1]{Lawson1991}}}]
${\bf C}(S)$ constructed as above equipped with $\leq_{r}$, $\leq_{l}$
is indeed an Ehresmann category. 
\end{prop}
The other direction is also possible. Given an Ehresmann category
$C$ we can construct an $E$-Ehresmann semigroup ${\bf S}(C)=S$
in the following way. The elements of $S$ are in one-to-one correspondence
with morphisms of $C$, for every $x\in C^{1}$ we associate an element
$S(x)\in S$. The distinguished semilattice is $E=\{S(x)\mid x\text{ is an identity morphism}\}$.
Note that $\leq_{r}=\leq_{l}$ on $E$ so we can denote the common
meet operation on $E$ simply by $\wedge$. The multiplication of
$S$ is defined by 
\begin{equation}
S(x)\cdot S(y)=S((x\mid\rb(x)\wedge\db(y))\cdot(\rb(x)\wedge\db(y)\mid y)).\label{eq:EhresmannProduct}
\end{equation}

\begin{rem}
\label{rem:EhresmannProduct}Note that if $\exists x\cdot y$ then
$S(x)\cdot S(y)=S(xy)$.\end{rem}
\begin{prop}[{{\cite[Theorem 4.21]{Lawson1991}}}]
\label{prop:ConstructionOfEhresmannSemigroup}${\bf S}(C)$ constructed
above is indeed an $E$-Ehresmann semigroup where for every $x\in C^{1}$
we have $(S(x))^{+}=S(\db(x))$ and $(S(x))^{\ast}=S(\rb(x))$. 
\end{prop}
The functions ${\bf C}$ and ${\bf S}$ are actually functors, moreover,
they are isomorphisms of categories. In order to state this theorem
accurately we need another definition. 
\begin{defn}
A functor $F:C\to D$ between two Ehresmann categories is called inductive
if the following hold: 
\begin{enumerate}
\item For every $x,y\in C^{1}$ we have that $x\leq_{r}y$ implies $F(x)\leq_{r}F(y)$
and $x\leq_{l}y$ implies $F(x)\leq_{l}F(y)$. 
\item $F(e\wedge f)=F(e)\wedge F(f)$ for every $e,f\in C^{0}$. 
\end{enumerate}
\end{defn}
In the following theorem, by a homomorphism of Ehresmann semigroups
we mean a $(2,1,1)$-algebra homomorphism, that is, a function that
preserves also the unary operations. 
\begin{thm}[{{\cite[Theorem 4.24]{Lawson1991}}}]
\label{thm:EhresmannIsoTheorem} The category of all $E$-Ehresmann
semigroups and homomorphisms is isomorphic to the category of all
Ehresmann categories and inductive functors. The isomorphism being
given by the functors ${\bf S}$ and ${\bf C}$ defined above.\end{thm}
\begin{rem}
We neglect the description of the operation of ${\bf S}$ and ${\bf C}$
on morphisms since it will be inessential in the sequel. 
\end{rem}
Let $S$ be an $E$-Ehresmann semigroup and let $C={\bf C(}S)$ be
the associated Ehresmann category (hence $S={\bf S}(C)$ by \thmref{EhresmannIsoTheorem}).
Some points about the correspondence between $S$ and $C$ are worth
mentioning. We will continue to denote by $C(a)$ the morphism in
$C$ associated to some $a\in S$ and likewise $S(x)$ is the element
of $S$ associated to some $x\in C^{1}$. In particular, $S(C(a))=a$
and $C(S(x))=x$. Two partial orders denoted by $\leq_{r}$ were defined
above, one on $S$ and one on $C^{1}$. Since $a\leq_{r}b$ if and
only if $C(a)\leq_{r}C(b)$ we can identify these partial orders so
the identical notation is justified. A dual remark holds for $\leq_{l}$.
The next lemma identifies the elements of $S$ corresponding to restriction
and co-restriction. 
\begin{lem}
\label{lem:RestrictionCoRestrictionElements} Let $a\in S$ and $e\in E$
then 
\[
C(ea)=(C(ea^{+})\mid C(a))
\]
\[
C(ae)=(C(a)\mid C(ea^{\ast})).
\]
\end{lem}
\begin{proof}
It is clear that $ea\leq_{r}a$ so $C(ea)\leq_{r}C(a)$. Moreover,
$(ea)^{+}=(ea^{+})^{+}=ea^{+}$. So $\db(C(ea))=C(ea^{+})$. By \enuref{EhresmannCatrestriction},
$(C(ea^{+})\mid C(a))$ is the unique morphism with these two properties
so the desired equality follows. The proof for $ae$ is similar. 
\end{proof}

\subsection{M\"{o}bius functions}

Let $(X,\leq)$ be a locally finite poset and let $\mathbb{K}$ be a commutative unital ring. Recall that locally finite
means that all the intervals $[x,y]=\{z\mid x\leq z\leq y\}$ are
finite.\textcolor{black}{{} We view $\leq$ as a set of ordered pairs.}
The \emph{M\"{o}bius function} of $\leq$ is a function $\mu:\leq\to\mathbb{K}$
that can be defined in the following recursive way: 
\[
\mu(x,x)=1
\]
\[
\mu(x,y)=-\sum_{x\leq z<y}\mu(x,z)
\]

\begin{thm}[M\"{o}bius inversion theorem]

Let $G$ be an abelian group and let $f,g:X\to G$ be functions such
that 
\[
g(x)=\sum_{y\leq x}f(y)
\]

then 
\[
f(x)=\sum_{y\leq x}\mu(y,x)g(y).
\]
\end{thm}

More on M\"{o}bius functions can be found in \cite[Chapter 3]{Stanley1997}.

\section{Isomorphism of algebras}

Throughout this section we let $S$ denote an $E$-Ehresmann semigroup and $C$ denote the corresponding Ehresmann category. For the sake of simplicity, we set $\leq=\leq_{r}$. From now on we
assume that for any $e \in E$ the set $\{f\in E\mid f\leq e\}$ is finite. This clearly implies that for every  $a\in S$ the set $\{b\in S\mid b\leq a\}$ is
finite. In this section we will prove that the semigroup algebra of $S$ over a commutative ring $\mathbb{K}$ with identity is
isomorphic to the algebra of $C$ over $\mathbb{K}$. This result is a generalization
of \cite[Theorem 4.2]{Steinberg2006} where it was proved for inverse
semigroups and inductive groupoids and of \cite[Theorem 4.2]{Guo2012}
where it was proved for ample semigroups. This also generalizes \cite[Proposition 3.2]{Stein2016}
where this isomorphism was proved for the special case $S=\PT_{n}$ (actually $\PT_n^{\op}$, since there the composition is from right to left).
We start by recalling the definition of an algebra of a semigroup or a category. 
\begin{defn}
Let $S$ be a semigroup. The\emph{ semigroup algebra } $\mathbb{K}S$
is the free $\mathbb{K}$-module with basis the elements of the semigroup.
In other words, as a set $\mathbb{K}S$ is all the formal linear combinations
\[
\{k_{1}s_{1}+\ldots+k_{n}s_{n}\mid k_{i}\in\mathbb{K},\,s_{i}\in S\}
\]
with multiplication being linear extension of the semigroup multiplication. 
\end{defn}

\begin{defn}
Let $C$ be a category. The \emph{category algebra} $\mathbb{K}C$
is the free $\mathbb{K}$-module with basis the morphisms of the category.
In other words, as a set $\mathbb{K}C$ is all formal linear combinations
\[
\{k_{1}x_{1}+\ldots+k_{n}x_{n}\mid k_{i}\in\mathbb{K},\,x_{i}\in C^{1}\}
\]
with multiplication being linear extension of 
\[
x\cdot y=\begin{cases}
xy & \exists x\cdot y\\
0 & \text{otherwise.}
\end{cases}
\]
\end{defn}
\begin{rem}
We use the word "algebra" in two different meaning in this paper, in
the sense of ring theory as in the above definitions and
in the sense of universal algebra. However, when we use it in the
later sense we always mention the signature: $(2,1,1)$-algebras,
$(2,1)$-subalgebras etc. Hence no ambiguity should arise.\end{rem}
\begin{thm}
\label{thm:MainThm}Let $S$ be an $E$-Ehresmann semigroup and denote
$C={\bf C}(S)$. Then $\mathbb{K}S$ is isomorphic to $\mathbb{K}C$.
Explicit isomorphisms $\varphi:\mathbb{K}S\rightarrow\mathbb{K}C$,
$\psi:\mathbb{K}C\rightarrow\mathbb{K}S$ are defined (on basis elements)
by 
\[
\varphi(a)=\sum_{b\leq a}C(b)
\]

\[
\psi(x)=\sum_{y\leq x}\mu(y,x)S(y)
\]

where $\mu$ is the M\"{o}bius function of the poset $\leq$. 
\end{thm}
Note that by our assumption the number of $b\in S$ such that $b\leq a$
is finite so the summations in \thmref{MainThm} are also finite.
Hence, $\varphi$ and $\psi$ are well defined. 
\begin{proof}
The proof that $\varphi$ and $\psi$ are bijectives is identical
to what is done in \cite{Steinberg2006}. 
\begin{align*}
\psi(\varphi(a)) & =\psi(\sum_{b\leq a}C(b))=\sum_{b\leq a}\psi(C(b))\\
 & =\sum_{b\leq a}\sum_{c\leq b}\mu(c,b)S(C(c))=\sum_{c\leq a}c\sum_{c\leq b\leq a}\mu(c,b)\\
 & =\sum_{c\leq a}c\delta(c,a)=a
\end{align*}
and 
\[
\varphi\psi(x)=\varphi(\sum_{y\leq x}\mu(y,x)S(y))=\sum_{y\leq x}\mu(y,x)\varphi(S(y))=C(S(x))=x
\]
where the third equality follows from the M\"{o}bius inversion theorem
and the definition of $\varphi$. Hence, $\varphi$ and $\psi$ are
bijectives. We now prove that $\varphi$ is a homomorphism. Let $a,b\in S$,
we have to prove that 
\begin{equation}
\sum_{c\leq ab}C(c)=(\sum_{a^{\prime}\leq a}C(a^{\prime}))(\sum_{b^{\prime}\leq b}C(b^{\prime})).\label{eq:HomomorphismEqualityOfVarphi}
\end{equation}

\begin{casenv}
\item \label{case:MainIsomorphismProofCase1}First assume that $\exists C(a)\cdot C(b)$,
that is, $\rb(C(a))=\db(C(b))$ (or equivalently, $a^{\ast}=b^{+}$).
In this case we can set $x=C(a)$ and $y=C(b)$ and then $C(ab)=C(a)C(b)=xy$.
So we can write \Eqref{HomomorphismEqualityOfVarphi} as 
\begin{equation}
\sum_{z\leq xy}z=(\sum_{x^{\prime}\leq x}x^{\prime})(\sum_{y^{\prime}\leq y}y^{\prime}).\label{eq:HomomorphismEqualityInMorphisms}
\end{equation}
According to \ref{enu:CatWithOrderCompatability} if $\exists x^{\prime}\cdot y^{\prime}$
then $x^{\prime}y^{\prime}\leq xy$. Hence, any element on the right
hand side of \Eqref{HomomorphismEqualityInMorphisms} is less than
or equal to $xy$. So we have only to show that any $z$ such that
$z\leq xy$ appears on the right hand side once. First, note that
$z=(\db(z)\mid xy)$ according to the uniqueness of restriction (part
of \ref{enu:EhresmannCatrestriction}). We can choose $x^{\prime}=(\db(z)\mid x$)
and $y^{\prime}=(\rb(x^{\prime})\mid y)$. Clearly, since $\db(y^{\prime})=\rb(x^{\prime})$
we have that $\exists x^{\prime}\cdot y^{\prime}$. Moreover by \ref{enu:CatWithOrderCompatability}
$x^{\prime}\cdot y^{\prime}\leq xy$ and $\db(x^{\prime}y^{\prime})=\db(x^{\prime})=\db(z)$
hence by uniqueness of restriction we have that $x^{\prime}\cdot y^{\prime}=(\db(z)\mid xy)=z$.
This proves that $z$ appears in the right hand side of \Eqref{HomomorphismEqualityInMorphisms}.
Now assume that $x^{\prime}\cdot y^{\prime}=z$ for some $x^{\prime}\leq x$
and $y^{\prime}\leq y$. Then we must have $\db(x^{\prime})=\db(z$)
so by uniqueness of restriction $x^{\prime}=(\db(z)\mid x$). Now,
since $\exists x^{\prime}\cdot y^{\prime}$ we must have that $\db(y^{\prime})=\rb(x^{\prime})$
so again by uniqueness of restriction $y^{\prime}=($$\rb(x^{\prime})\mid y)$.
So $z$ appears only once on the right hand side of \Eqref{HomomorphismEqualityInMorphisms}
and this finishes this case. 
\item Assume $\rb(C(a))\neq\db(C(b))$ (or equivalently, $a^{\ast}\neq b^{+})$.
Define $\tilde{a}=ab^{+}$ and $\tilde{b}=a^{\ast}b$. Note that 
\[
\tilde{a}\tilde{b}=ab^{+}a^{\ast}b=aa^{\ast}b^{+}b=ab
\]
so we have 
\[
\sum_{c\leq ab}C(c)=\sum_{c\leq\tilde{a}\tilde{b}}C(c).
\]
By \lemref{RestrictionCoRestrictionElements} 
\[
C(\tilde{a})=(C(a)\mid C(a^{\ast}b^{+}))=(C(a)\mid\rb(C(a))\wedge\db(C(b)))
\]
and 
\[
C(\tilde{b})=(C(a^{\ast}b^{+})\mid C(b))=(\rb(C(a))\wedge\db(C(b))\mid C(b))
\]
so clearly $\exists C(\tilde{a})\cdot C(\tilde{b})$. \caseref{MainIsomorphismProofCase1}
implies that 
\[
\sum_{c\leq\tilde{a}\tilde{b}}C(c)=(\sum_{a^{\prime}\leq\tilde{a}}C(a^{\prime}))(\sum_{b^{\prime}\leq\tilde{b}}C(b^{\prime})).
\]
Now, all that is left to show is that 
\begin{equation}
(\sum_{a^{\prime}\leq\tilde{a}}C(a^{\prime}))(\sum_{b^{\prime}\leq\tilde{b}}C(b^{\prime}))=(\sum_{a^{\prime}\leq a}C(a^{\prime}))(\sum_{b^{\prime}\leq b}C(b^{\prime})).\label{eq:MultiplicationOfLinearCombinationWithTilde}
\end{equation}
We can set again $x=C(a)$, $\tilde{x}=C(\tilde{a})$, $y=C(b)$ and
$\tilde{y}=C(\tilde{b})$ so \Eqref{MultiplicationOfLinearCombinationWithTilde}
can be written as 
\begin{equation}
(\sum_{x^{\prime}\leq\tilde{x}}x^{\prime})(\sum_{y^{\prime}\leq\tilde{y}}y^{\prime})=(\sum_{x^{\prime}\leq x}x^{\prime})(\sum_{y^{\prime}\leq y}y^{\prime}).\label{eq:MultiplicationWithTildesInMorphisms}
\end{equation}
We will show that a multiplication $x^{\prime}\cdot y^{\prime}$ on
the right hand side of \Eqref{MultiplicationWithTildesInMorphisms}
equals $0$ unless $x^{\prime}\leq\tilde{x}$ and $y^{\prime}\leq\tilde{y}$.
Take $x^{\prime}\leq x$ such that $x^{\prime}\nleq\tilde{x}$ and
assume that there is a $y^{\prime}\leq y$ such that $\exists x^{\prime}\cdot y^{\prime}$,
that is, $\rb(x^{\prime})=\db(y^{\prime})$. Since $y^{\prime}\leq y$
we have $\rb(x^{\prime})=\db(y^{\prime})\leq\db(y)$ by \enuref{CatWithOrderLeqOfDomRan}.
Now, by \ref{enu:EhresmannCatWedgeCorestriction} (choosing $f=\db(y)$)
we have that 
\[
(x^{\prime}\mid\rb(x^{\prime})\wedge\db(y))\leq(x\mid\rb(x)\wedge\db(y))
\]
but note that $(x\mid\rb(x)\wedge\db(y))=\tilde{x}$ and $\rb(x^{\prime})\wedge\db(y)=\rb(x^{\prime})$
so we get 
\[
x^{\prime}=(x^{\prime}\mid\rb(x^{\prime}))\leq\tilde{x}
\]
a contradiction. Similarly, take $y^{\prime}\leq y$ such that $y^{\prime}\nleq\tilde{y}$
and assume that there is an $x^{\prime}\leq x$ such that $\exists x^{\prime}\cdot y^{\prime}$,
that is, $\rb(x^{\prime})=\db(y^{\prime})$. Again, since $\rb(x^{\prime})\leq\rb(x)$
we have that $\db(y^{\prime})\leq\rb(x)$ and clearly $\db(y^{\prime})\leq\db(y)$
hence $\db(y^{\prime})\leq\rb(x)\wedge\db(y)=\db(\tilde{y})$. By
\ref{enu:EhresmannCatrestriction} there exists a restriction $(\db(y^{\prime})\mid\tilde{y})$.
But $(\db(y^{\prime})\mid\tilde{y})\leq\tilde{y}\leq y$ so by the
uniqueness of restriction $(\db(y^{\prime})\mid\tilde{y})=y^{\prime}$, hence $y^{\prime}\leq\tilde{y}$, a contradiction. This finishes the
proof. 
\end{casenv}
\end{proof}
\begin{rem}
Note that \thmref{MainThm} can be proved, \emph{mutatis mutandis}, using $\leq_l$ instead of $\leq_r$.
\end{rem}
\begin{cor}
Let $S$ be an $E$-Ehresmann semigroup such that $E$ is finite,
then $\mathbb{K}S$ is a unital algebra.\end{cor}
\begin{proof}
The isomorphic category algebra $\mathbb{K}C$ has the identity element
${\displaystyle \sum_{e\in E}C(e)}$. 
\end{proof}

\section{\label{sec:Examples}Examples}

In the following examples $C$ will always be the Ehresmann category
associated to the $E$-Ehresmann semigroup being discussed. 
\begin{example}
Let $M$ be a monoid and take $E=\{1\}$. It is easy to check that
$M$ is an $E$-Ehresmann semigroup. It is easy to see that if we
think of $M$ as a category with one object in the usual way we get
precisely $C$. The fact that $\mathbb{K}M$ is isomorphic to $\mathbb{K}C$
is trivial but true. 
\end{example}

\begin{example}
Let $S$ be an inverse semigroup such that $E(S)$ is finite. If we
take $E=E(S)$ our isomorphism is precisely \cite[Theorem 4.2]{Steinberg2006}.
If $S$ is a finite ample semigroup our isomorphism is precisely \cite[Theorem 4.2]{Guo2012}. 
\end{example}

\begin{example}
\label{exa:PartialFunctionsExample}Let $S=\PT_{n}$ be the monoid
of all partial functions on an $n$-element set and take $E=\{1_{A}\mid A\subseteq\{1\ldots n\}\}$
to be the semilattice of all the partial identities. It can be checked
that $\PT_{n}$ is an $E$-Ehresmann semigroup where for every $t\in\PT_{n}$
we have $t^{+}=1_{\dom(t)}$ and $t^{\ast}=1_{\im(t)}$. The corresponding
Ehresmann category is the category of all onto (total) functions between
subsets of an $n$-element set. Our isomorphism is then precisely
\cite[Proposition 3.2]{Stein2016}.
\end{example}

\begin{example}
\label{exa:BinaryRelationsExamples}Let $S=B_{n}$ be the monoid of
all relations on an $n$-element set and take again $E$ to be the
semilattice of all the partial identities. Again, $B_{n}$ is an $E$-Ehresmann
semigroup where for every $t\in B_{n}$ we have $t^{+}=1_{\dom(t)}$
and $t^{\ast}=1_{\im(t)}$. The associated category $C$ has the subsets
of $\{1,\ldots,n\}$ as objects and for every $a\in B_{n}$ there
is a corresponding morphism $C(a)$ from $\dom(a)$ to $\im(a)$.
This is the category of bi-surjective relations on the subsets of
an $n$-element set $X$. That is, the objects are all subsets of
$X$ and a morphism from $Y$ to $Z$ are all subsets $R$ of $Y\times Z$
such that both projections of $R$ to $Y$ and $Z$ respectively are
onto functions. This is a subcategory of the category applied in \cite{bouc2015}
to find the dimensions of the simple modules of $B_{n}$. 
\end{example}

\begin{example}
Let $S=[Y,M_{\alpha},\varphi_{\alpha,\beta}]$ be a strong semilattice
of monoids (where $Y$ is a finite semilattice). If we take $E=\{1_{\alpha}\in M_{\alpha}\mid\alpha\in Y\}\cong Y$
then it is proved in \cite[Examples 2.5.11-12]{Cornock2011} that
$S$ is an $E$-Ehresmann semigroup. In this case, the objects of
$C$ are in one-to-one correspondence with elements of $Y$. Every
$a\in M_{\alpha}$ corresponds to an endomorphism $C(a)$ of $\alpha$.
Note that all the morphisms in $C$ are endomorphisms.\end{example}
\begin{cor}
If $S=[Y,M_{\alpha},\varphi_{\alpha,\beta}]$ is a strong
semilattice of monoids with a finite $Y$ then $\mathbb{K}S$ is isomorphic
to ${\displaystyle \prod_{\alpha\in Y}}\mathbb{K}M_{\alpha}$. 
\end{cor}

\section{Ehresmann EI-categories}

\subsection{\label{sec:Ehresmann-EI-categories}Characterization and examples}

Recall that if $C$ is a category an \emph{endomorphism} is a morphism of $C$ from some object $e\in C^0$ to $e$.
A category is called an \emph{EI-category} if every endomorphism is
an isomorphism, or in other words, if every endomorphism monoid is
a group. Algebras of EI-categories are better understood than general
category algebras. Assume $\mathbb{K}$ is an algebraically closed field and $C$ is a finite
EI-category where the orders of the endomorphism groups of its objects are invertible
in $\mathbb{K}$. There is a way to describe the Jacobson radical
of $\mathbb{K}C$ (\cite[Proposition 4.6]{Li2011}) and its ordinary quiver
(\cite[Theorem 4.7]{Li2011} or \cite[Theorem 6.13]{Margolis2012}).
By \thmref{MainThm} we then know how to compute the Jacobson radical
and ordinary quiver of an $E$-Ehresmann semigroup if the corresponding
Ehresmann category is an EI-category. In this section we will characterize
such semigroups. We start with understanding elements corresponding
to left, right and two-sided invertible morphisms. Recall that if
$a\in S$ then $\db(C(a))=C(a^{+})$ and $\rb(C(a))=C(a^{\ast})$. 
\begin{lem}
\label{lem:CharacterizeRightInvertibility}$C(a)$ is right invertible
if and only if $a\Rc a^{+}$ and $a$ has an inverse $b$ such that
$b\Lc a^{+}$ and $a^{\ast}=b^{+}$. A dual statement holds for left
invertibility.\end{lem}
\begin{proof}
Assume that $C(b)$ is a right inverse for $C(a)$, then clearly $\db(C(a))=\rb(C(b))$
and $\rb(C(a))=\db(C(b))$ so $a^{+}=b^{\ast}$ and $a^{\ast}=b^{+}$.
Since $C(ab)=C(a)C(b)=\db(C(a))=C(a^{+})$ we know that 
\[
ab=a^{+}
\]
and clearly 
\[
a^{+}a=a,\quad ba^{+}=bb^{\ast}=b.
\]
Hence 
\[
a\Rc a^{+},\quad b\Lc a^{+}
\]
so $b$ is an inverse of $a$ as required.

In the other direction, we have that $\exists C(a)\cdot C(b)$ since $a^{\ast}=b^{+}$. Now $C(a)\cdot C(b)=C(ab)=C(a^{+})$
since $b$ is an inverse of $a$ with $a\Rc a^{+}\Lc b$. This finishes
the proof since the dual case is similar.
\end{proof}

\begin{rem}
The requirement $a\Rc a^{+}$($a\Lc a^{\ast}$) is not enough for
right (respectively, left) invertibility. For instance take $S=\PT_{2}$
and $E=\{1_{A}\mid A\subseteq\{1,2\}\}$ as above. Choose $a$ to
be the (total) constant transformation with image $\{1\}$.

\[
a=\left(\begin{array}{cc}
1 & 2\\
1 & 1
\end{array}\right).
\]

It is easy to check that $C(a)$ is not left invertible in $C$ but
it is $\Lc$-equivalent to

\[
a^{\ast}=\left(\begin{array}{cc}
1 & 2\\
1 & \emptyset
\end{array}\right).
\]

\end{rem}
However, we have the following two sided version. 
\begin{lem}
\label{lem:InvertibleCriterion}$C(a)$ is invertible in $C$ if and
only if $a\Rc a^{+}$ and $a\Lc a^{\ast}$.\end{lem}
\begin{proof}
If $C(a)$ is left and right invertible then \lemref{CharacterizeRightInvertibility}
implies that $a\Rc a^{+}$ and $a\Lc a^{\ast}$. In the other direction
take $b$ to be the inverse of $a$ such that $b\Rc a^{\ast}$and $b\Lc a^{+}$
and again the result follows from \lemref{CharacterizeRightInvertibility}.\end{proof}
\begin{rem} \label{rem:OnSetReg}
The set of all elements $a\in S$ such that $a\Rc a^{+}$ and $a\Lc a^{\ast}$
appears in \cite[Section 3]{Lawson1990} in a more general context
and it is denoted by $\Reg_{E}(S)$. We will see later (\lemref{RegESubsemigroup}) that this is an inverse subsemigroup of $S$ \end{rem}
\begin{cor}
$C(e),C(f)\in C^{0}$ are isomorphic objects if and only if $e\Dc f$.\end{cor}
\begin{proof}
If $e\Dc f$ take $a\in\Rc_{e}\cap\Lc_{f}$ and $C(a)$ is an isomorphism
between $C(e)$ and $C(f)$. On the other hand, if $C(a)$ with $\db(C(a))=C(e)$
and $\rb(C(a))=C(f)$ is an isomorphism then $a\Rc e$ and $a\Lc f$
so $e\Dc f$. \end{proof}
\begin{cor}
$C$ is an $EI$-category if and only if $a^{+}=a^{\ast}$ implies
that $a$ is a group element. In other words, $C$ is an EI-category
if and only if the $\Ht$-class of $e$ equals its $\Hc$ class for
every $e\in E$.\end{cor}
\begin{proof}
Clear from \lemref{InvertibleCriterion}.
\end{proof}
We also note the following necessary condition for $C$ to be an EI-category.
\begin{lem}
\label{lem:EIImpliesMaximalSemilattice} If $C$ is an EI-category
then $E$ is a maximal semilattice in $S$. \end{lem}
\begin{proof}
Assume that there is some $f\in E(S)\backslash E$ which commutes
with every $e\in E$. It follows that $f^{+}=f^{\ast}$ so the morphism
$C(f)$ is an element in some endomorphism group of $C$. This is
a contradiction since $C(f)C(f)=C(f)$ and groups have no non-identity
idempotents.
\end{proof}

Note that the condition of \lemref{EIImpliesMaximalSemilattice} is
not sufficient. For instance, take the monoid of all binary relations
$B_{n}$ which is $E$-Ehresmann as mentioned above (\exaref{BinaryRelationsExamples}).
It is easy to check that the set of all partial identities is a maximal
semilattice but the corresponding category $C$ is not an EI-category.

It is also worth mentioning the groupoid case.
\begin{cor}
$C$ is a groupoid if and only if $S$ is an inverse semigroup and
$E=E(S)$.\end{cor}
\begin{proof}
Assume that $C$ is a groupoid. Let $a,b\in S$ such that $a\Rt b$.
By \lemref{InvertibleCriterion} 
\[
a\Rc a^{+}=b^{+}\Rc b
\]
hence $\Rc=\Rt$ and this implies that any $\Rc$ class contains precisely
one idempotent. A similar observation is true for $\Lc$ classes.
Hence $S$ is inverse and $E(S)$ is a semilattice. By \lemref{EIImpliesMaximalSemilattice}
$E$ is a maximal semilattice so $E=E(S)$ as required. The other
direction is clear from \lemref{InvertibleCriterion}.
\end{proof}
\begin{rem} Note that in this case $C$ is the inductive groupoid corresponding to $S$.
\end{rem}
We now give some examples of $E$-Ehresmann semigroups whose corresponding
Ehresmann category is an EI-category.
\begin{example}
\label{exa:SubalgebrasOfPTn}Take $S=\PT_{n}$ which is $E$-Ehresmann
as mentioned in \exaref{PartialFunctionsExample}. The corresponding
Ehresmann category $E_{n}$ is the category of all onto functions
between subsets of an $n$-element set. $E_n$ is an EI-category since the endomorphism monoid of an object $A$ is the group $S_A$ of all permutations of elements of $A$. Every $(2,1,1)$-subalgebra
of $\PT_{n}$ is also an $E$-Ehresmann semigroup (since $E$-Ehresmann
semigroups form a variety). The corresponding Ehresmann category $C$
is a subcategory of $E_{n}$. So every endomorphism monoid of $C$
is a submonoid of some endomorphism group in $E_{n}$. Since every
submonoid of a finite group is a group, $C$ is also an EI-category.
Several well-known examples of such subsemigroups include: order-preserving
partial functions (or more generally, order-preserving partial functions
with respect to some partial order on $\{1,\ldots,n\}$), weakly-decreasing
partial functions, order-preserving and weakly-decreasing partial
functions (also known as the partial Catalan monoid), and orientation-preserving
partial functions. In \cite[Section 5]{Stein2016} the author has
used the corresponding EI-category to describe the ordinary quiver
of the algebras of some of these semigroups. Another remark is worth
mentioning. $\PT_{n}$ satisfies the identity $xy^{+}=(xy^{+})^{+}x$.
Left $E$-Ehresmann semigroups that satisfy this identity are called
\emph{left restriction}. It is well known \cite[Corollary 6.3]{Gould2010} that left restriction semigroups
are precisely the $(2,1)$-subalgebras of $\PT_{n}$ where the unary
operation is $^{+}$. Clearly, every
$(2,1,1)$-subalgebra of $\PT_{n}$ is also left restriction.
\end{example}

\begin{example}
\label{exa:E=00003DE(S)}Let $S$ be a finite $E(S)$-Ehresmann semigroup,
i.e., an $E$-Ehresmann semigroup with $E=E(S)$. Let $a\in S$ be
an element such that $a^{+}=a^{\ast}$ and $C(a)$ is an idempotent
of the endomorphism monoid of $C(a^{+})$. Since $C(a)=C(a)C(a)=C(aa)$
we know that $a\in E(S)=E$ so $a=a^{+}$. Hence any idempotent of
an endomorphism monoid is the identity morphism of the object. A
finite monoid with only one idempotent is a group hence the corresponding
category is an EI-category. In particular this class contains all
finite adequate semigroups and hence all finite ample semigroups which
are the semigroups considered in \cite{Guo2012}.
\end{example}

\begin{example}
Let $\T_n$ be the monoid of all (total) functions on an $n$-element set. Denote by $\id$ the identity function and by $\bf{k}$ the constant functions that sends every element to $k$. Define $S$ to be the subsemigroup of $\T_2 \times \T_2^{\op}$ containing the six elements 
\[
{(\bf{1},\bf{1}),  (\bf{2},\bf{1}), (\bf{1},\bf{2}), (\bf{2},\bf{2}), (\bf{1},\id), (\id,\bf{1})}.
\]


Note that 
\[
B=\{(\bf{1},\bf{1}),  (\bf{2},\bf{1}), (\bf{1},\bf{2}), (\bf{2},\bf{2})\}
\] forms a rectangular band and that $(\bf{1},\id)$ ($(\id,\bf{1})$)
is a left (respectively, right) identity for elements in $B$. Choose $E=\{(\bf{1},\bf{1}),(\bf{1},\id),(\id,\bf{1})\}$, which is clearly a subsemilattice.
It is easy to check that every $\Rt$ and $\Lt$ class contains one element of $E$ and the identities $(xy)^{+}=(xy^{+})^{+}$ and
$(xy)^{\ast}=(x^{\ast}y)^{\ast}$ hold so this is an $E$-Ehresmann
semigroup. The corresponding Ehresmann category (which is clearly
an EI-category) is given in the following drawing:

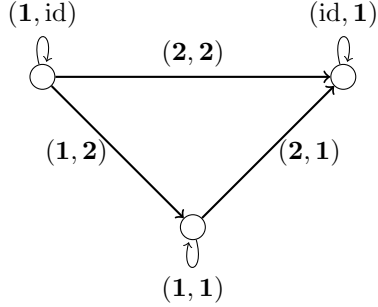
\begin{center}
\begin{tikzpicture}\path (0,2) node [shape=circle,draw] (e) {}  edge [loop above] node {$ (\bf{1},\id)$} (); \path (4,2) node [shape=circle,draw] (f) {} edge [loop above] node {$(\id,\bf{1})$} ();\path (2,0) node [shape=circle,draw] (a) {} edge [loop below] node {$(\bf{1},\bf{1})$} (); \draw[thick,->] (e)--(f) node [above,midway] {$ (\bf{2},\bf{2})$}; \draw[thick,->] (e)--(a) node [left,midway] {$ (\bf{1},\bf{2})$}; \draw[thick,->] (a)--(f) node [right,midway] {$ (\bf{2},\bf{1})$}; \end{tikzpicture}
\end{center}

Note that this example is not included in the previous ones. Definitely,
$E\neq S=E(S)$. Moreover, $S$ is not left or right restriction since
\[
(\bf{2},\bf{2})(\bf{1},\id)=(\bf{1},\bf{2})\neq (\bf{2},\bf{2})=(\bf{1},\id)(\bf{2},\bf{2})=((\bf{2},\bf{2})(\bf{1},\id))^{+}(\bf{2},\bf{2})
\]
and

\[
 (\id,\bf{1})(\bf{2},\bf{2})=(\bf{2},\bf{1})\neq (\bf{2},\bf{2})=(\bf{2},\bf{2}) (\id,\bf{1})=(\bf{2},\bf{2})((\id,\bf{1})(\bf{2},\bf{2}))^{\ast}.
\]
Hence, $S$ is not a $(2,1,1)$-subalgebra of $\PT_{n}$ (or $\PT_{n}^{\op}$).
The semigroup $S$ provides an example of an $E$-Ehresmann semigroup
whose corresponding Ehresmann category is an EI-category but there
is an idempotent $({\bf{2}},{\bf{2}})\in S$ which is not $\Rc$ equivalent to $(\bf{2},\bf{2})^{+}$
and not $\Lc$ equivalent to $(\bf{2},\bf{2})^{\ast}$.
\end{example}
\subsection{\label{sec:Max-Semisimple-Image}The maximal semisimple image}

Let $S$ be a finite semigroup whose idempotents $E(S)$ commute and let $\mathbb{K}$ be a field whose characteristic does not divide the order of any maximal subgroup of $S$. Denote by $R(S)$ the set of regular elements of $S$. It is known that $R(S)$ is an inverse subsemigroup of $S$. Steinberg proved \cite[Section 8.1]{Steinberg2008} that the algebra $\mathbb{K}R(S)$ is isomorphic to the maximal semisimple image of $\mathbb{K}S$. In this section we prove a similar result for finite $E$-Ehresmann semigroups whose corresponding Ehresmann category is an EI-category. More precisely, let $S$ be a finite $E$-Ehresmann semigroup with a corresponding
finite Ehresmann category $C$ such that $C$ is an EI-category. Assume that $\mathbb{K}$
is a field such that the orders of the endomorphism groups of all objects in $C$
are invertible in $\mathbb{K}$. Recall that we denote by $\Reg_E(S)$ (see \remref{OnSetReg}) the set of elements of $S$ corresponding to isomorphisms in $C$. We will prove that $\Reg_E(S)$  is an inverse subsemigroup of $S$ and $\mathbb{K}\Reg_E(S)$ is isomorphic to the maximal semisimple image of $\mathbb{K}S$.

We start by recalling some elementary definitions and facts about algebras. Recall
that a non-zero module is called \emph{simple} if it has no non-trivial
submodules. A module is called \emph{semisimple} if it is a direct sum of
simple modules. An algebra $A$ is called \emph{semisimple} if all its modules
are semisimple. The \emph{Jacobson radical} of an algebra $A$ denoted $\Rad A$
is the intersection of all its left maximal ideals. Its importance comes
from the fact that if $A$ is finite dimensional then $A/\Rad(A)$ is the maximal semisimple image of
$A$. Other fundamental facts about algebras can be found in the first
chapter of \cite{Assem2006}.

Let $A$ be a finite dimensional associative algebra. It might be
the case that the maximal semisimple image $A/\Rad A$ is isomorphic
to a subalgebra of $A$. For instance, if $A/\Rad A$ is a separable algebra then the Wedderburn-Malcev theorem \cite[Theorem 72.19]{Curtis1966}
assures that $A/\Rad A$ is isomorphic to a subalgebra
of $A$. However, if $A$ is an algebra of a category or a semigroup,
it is usually not the case that an isomorphic copy of $A/\Rad A$
is spanned by some subcategory or a subsemigroup. However, if $A$
an EI-category algebra the situation is better.
\begin{prop}
\cite[Proposition 4.6]{Li2011} Let $C$ be a finite EI-category and let $\mathbb{K}$ be a field such that the orders
of the endomorphism groups of all objects in $C$ are invertible in $\mathbb{K}$.
The radical $\Rad(\mathbb{K}C)$ is spanned by all the non-invertible
morphisms of $C$.\end{prop}
\begin{cor}
\label{cor:SemisimplePartOfEICategoryAlgebra}Let $C$ and $\mathbb{K}$
be as above. The subalgebra of $\mathbb{K}C$ spanned by all the invertible
morphisms is isomorphic to the maximal semisimple image $\mathbb{K}C/\Rad(\mathbb{K}C)$.
\end{cor}
We now want to ``translate'' this result to semigroup language
using \thmref{MainThm}.
Recall that we denote by $\Reg_{E}(S)$ the elements of $S$ that
correspond to invertible morphisms in $C$. According to \lemref{InvertibleCriterion},
\[
\Reg_{E}(S)=\{a\in S\mid a^{+}\Rc a\Lc a^{\ast}\}.
\]

A natural question is whether this set forms a subsemigroup. This question is also considered in other contexts in \cite{Lawson1990,Zenab2013}. The next lemma gives an affirmative answer in the case of $E$-Ehresmann semigroups.

\begin{lem}
\label{lem:RegESubsemigroup}Let $S$ be an $E$-Ehresmann semigroup. Then the set $\Reg_{E}(S)$ is an inverse subsemigroup of $S$.\end{lem}
\begin{proof}
We first prove that $\Reg_{E}(S)$  forms a subsemigroup\footnote{This part of the proof is actually due to Michael Kinyon who proved it using the automated theorem prover called Prover9. For general information on Prover9 see \cite{Prover9}}. Let $a$ and $b$ be two elements of $\Reg_E(S)$. We will prove only $ab \Rc (ab)^{+}$ because the proof that $ab \Lc (ab)^{\ast}$ is similar. For every $x \in \Reg_E(S)$ it is convenient to denote by $x^{-1}$ the unique inverse of $x$ such that $xx^{-1}=x^{+}$ and $x^{-1}x=x^{\ast}$.
Clearly,
\[
(ab)^{+}ab=ab.
\]
so we need only to prove that $ab$ is $\Rc$-above $(ab)^{+}$. Now, note that
\begin{align*}
a^{-1}(ab)^{+} &= (a^{-1}(ab)^{+})^{+}a^{-1}(ab)^{+}=(a^{-1}ab)^{+}a^{-1}(ab)^{+} \\
&= (a^{\ast}b)^{+}a^{-1}(ab)^{+}=a^{\ast}b^{+}a^{-1}(ab)^{+}.
\end{align*}
Multiplying by $a$ on the left we get
\[
a^{+}(ab)^{+}=ab^{+}a^{-1}(ab)^{+}=abb^{-1}a^{-1}(ab)^{+}.
\]
But since $(ab)^{+}$ is the least element from $E$ which is left identity for $ab$ and since
\[
a^{+}ab=ab
\]
we have that 
\[
(ab)^{+}=a^{+}(ab)^{+}=abb^{-1}a^{-1}(ab)^{+}.
\]
This implies that $ab \Rc (ab)^{+}$ so $\Reg_E(S)$ is indeed a subsemigroup.
It is left to show that $\Reg_E(S)$ is an inverse subsemigroup. Clearly, $E\subseteq\Reg_{E}(S)$ and all the elements
of $\Reg_{E}(S)$ are regular. Moreover, any idempotent $f\in\Reg_{E}(S)$
must be in $E$. Indeed, if $f^{+}\Rc f\Lc f^{\ast}$ then $f^{\ast}f^{+} \Dc f$ by the Clifford-Miller theorem \cite[Proposition 2.3.7]{Howie1995} and it follows that $f^{\ast}$
and $f^{+}$ commute only if $f^{\ast}=f^{+}$. In this case
$f\Hc f^{+}$ so $f=f^{+}\in E$. So $\Reg_{E}(S)$  is regular and its set of idempotents is the semilattice $E$,
hence it is an inverse semigroup as required.
\end{proof}

\begin{cor}
\label{cor:RegEDownIdeal}
$\Reg_{E}(S)$ is a down ideal with respect to $\leq_{r}$ and $\leq_{l}$.
\end{cor}
\begin{proof}
If $a\in \Reg_E(s)$ and $b \leq _r a$ ($b \leq _l a$) then $b=ea$ ($b=ae$) for some $e\in E$. \lemref{RegESubsemigroup} then implies that  $b\in \Reg_E(S)$. 
\end{proof}

\begin{prop}
Let $S$ be a finite $E$-Ehresmann semigroup whose corresponding Ehresmann
category $C$ is an EI-category. Then $\mathbb{K}\Reg_{E}(S)$ is isomorphic to $\mathbb{K}S/\Rad(\mathbb{K}S)$.
\end{prop}
\begin{proof}
Denote by $B$ the subvector space of $\mathbb{K}C$
spanned by all the invertible morphisms. Clearly, $B$ is a subalgebra
of dimension $|\Reg_{E}(S)|$, since the product in $\mathbb{K}C$ of two invertible morphisms in $C$ is either another invertible morphism or $0$. Denote by $\psi$ the isomorphism $\psi:\mathbb{K}C\to\mathbb{K}S$
as in \thmref{MainThm}. Since 
\[
\psi(C(a))=\sum_{b\leq_{r}a}\mu(b,a)b
\]
and since $\Reg_{E}(S)$ is a down ideal with respect to $\leq_{r}$ by \corref{RegEDownIdeal}
we have that $\psi(B)\subseteq\mathbb{K}\Reg_{E}(S)$. But since they
have the same (finite) dimension we must have $\psi(B)=\mathbb{K}\Reg_{E}(S)$. Now, by \corref{SemisimplePartOfEICategoryAlgebra}
$B$ is isomorphic to the maximal semisimple image of $\mathbb{K}C$
so $\psi(B)=\mathbb{K}\Reg_{E}(S)$ is isomorphic to the maximal semisimple
image $\mathbb{K}S/\Rad(\mathbb{K}S)$ of $\mathbb{K}S$ as required.
\end{proof}

\textbf{Acknowledgements: } The author is grateful to Michael Kinyon for proving that $\Reg_E(S)$ is a subsemigroup (the main part of  \lemref{RegESubsemigroup}). The author also thanks the referee for his\textbackslash her helpful comments.

\bibliographystyle{plain}
\bibliography{library}

\end{document}


\title{Erratum to: Algebras of Ehresmann semigroups and categories}

\author{Itamar Stein\\
Department of Mathematics\\
Bar Ilan University\\
Israel\\
{\em steinita@gmail.com}}

\maketitle
Shoufeng Wang discovered an error in the main theorem  of the author's Semigroup Forum article
`Algebras of Ehresmann semigroups and categories'. Wang observed
that the function we suggest as an isomorphism is not a homomorphism
unless the semigroup being discussed is left restriction. In order
to fix our mistake we will add this assumption.
Note that our revised result is still a generalization of earlier 
work of Guo and Chen, the author, and Steinberg .

\section{A correction to the main theorem of \cite{Stein2016Ehresmann}}

Shoufeng Wang \cite{Wang:2017} has observed that the proof of \cite[Theorem 3.4]{Stein2016Ehresmann} does not hold without the additional assumption that the semigroup is
left restriction. This Erratum shows how this assumption yields  a valid result, and examines the consequences. Theorem~\ref{thm:MainThm} below, which is the revision of  \cite[Theorem 3.4]{Stein2016Ehresmann},   nevertheless generalizes  
 \cite[Theorem 4.2]{Steinberg2006},
\cite[Theorem 4.2]{Guo2012} and \cite[Proposition 3.2]{Stein2016}.

We assume the reader is familiar with  \cite{Stein2016Ehresmann}
and in particular with the definition of an $E$-Ehresmann semigroup; for undefined terms the reader should consult \cite{Stein2016Ehresmann}.

We start by giving a counterexample to the original claim in
\cite[Theorem 3.4]{Stein2016Ehresmann}. Let $S$ is an $E$-Ehresmann
semigroup  with $\leq=\leq_{r}$  a principally finite poset, and
let $C=\mathbf{C}(S)$ be the corresponding Ehresmann category. We
do not know whether $\mathbb{K}S$ is always isomorphic to $\mathbb{K}C$,
but Shoufeng Wang observed that the function $\varphi:\mathbb{K}S\to\mathbb{K}C$
defined by
\[
\varphi(a)=\sum_{b\leq a}C(b)
\]

is, in general, not a homomorphism, as can be seen in the following
 example.
\begin{example}
Choose $S=B_{2}$, the monoid of all binary relations on the set $\{1,2\}$.
As mentioned in \cite[Example 4.4]{Stein2016Ehresmann}, $B_{2}$
is an $E$-Ehresmann semigroup where $E=\{\id,\{(1,1)\},\{(2,2)\},\varnothing\}$
is the set of partial identities. For every $a\in B_{2}$, $a^{+}$
($a^{\ast}$) is the identity function on the domain (respectively,
image) of $a$. Let $C={\bf C}(B_{2})$ be the corresponding Ehresmann
category. It is easy to see that $\varphi:\mathbb{K}B_{2}\to\mathbb{K}C$
is not a homomorphism. Choose $a=\{(1,1),(1,2)\}$ and $b=\{(1,1)\}$
so $ab=\{(1,1)\}$. Note that in $B_{2}$, $\leq$ is \emph{domain}
restriction, so
\[
\varphi(b)=\varphi(ab)=\varphi(\{(1,1)\})=C(\{(1,1)\})+C(\varnothing)
\]
and
\[
\varphi(a)=C(\{(1,1),(1,2)\})+C(\varnothing).
\]

Recall that in the category algebra the multiplication $xy$ equals
$0$ unless $\rb(x)=\db(y)$. Hence 
\[
\varphi(a)\varphi(b)=\left(C(\{(1,2),(1,1)\})+C(\varnothing)\right)\left(C(\{(1,1)\})+C(\varnothing)\right)=C(\varnothing)
\]
so indeed
\[
\varphi(ab)\neq\varphi(a)\varphi(b).
\]
\end{example}
In order to fix this problem we will have to add the requirement of
being left (or right) restriction. 
\begin{defn}
Let $S$ be a left $E$-Ehresmann semigroup. $S$ is called \emph{left
restriction} if 
\[
ae=(ae)^{+}a
\]
for every $a\in S$ and $e\in E$. Dually, a right $E$-Ehresmann
semigroup $S$ is called \emph{right restriction} if 
\[
ea=a(ea)^{\ast}
\]

for every $a\in S$ and $e\in E$. If $S$ is an $E$-Ehresmann semigroup
which is both left and right restriction, then it is called a \emph{restriction}
semigroup.
\end{defn}
\begin{example}
Every inverse semigroup is a restriction semigroup. It is well-known
that the monoid $\PT_{n}$ of all partial transformations on an $n$-element set
is left restriction but not right-restriction. However, the monoid
$B_{n}$ of all binary relations on an $n$-element set is neither left nor
right restriction. 
\end{example}
\begin{lem}
\label{lem:LeqLInLeqRForRestriction}If $S$ is an $E$-Ehresmann
semigroup which is also left restriction (right restriction), then
$\leq_{l}\subseteq\leq_{r}$ (respectively, $\leq_{r}\subseteq\leq_{l}$).
\end{lem}
\begin{proof}
Assume $b\leq_{l}a$. By \cite[Proposition 2.6]{Stein2016Ehresmann}
there exists an $e\in E$ such that $b=ae=(ae)^{+}a$ but $(ae)^{+}\in E$
so $b\leq_{r}a$ as well. The other case is dual.
\end{proof}

We can now give a correct version of \cite[Theorem 3.4]{Stein2016Ehresmann},
under the additional assumption of being left restriction.

\begin{thm}
\label{thm:MainThm}Let $S$ be an $E$-Ehresmann and left restriction
semigroup and let $C={\bf C}(S)$. Then $\mathbb{K}S$ is isomorphic
to $\mathbb{K}C$. Explicit isomorphisms $\varphi:\mathbb{K}S\rightarrow\mathbb{K}C$,
$\psi:\mathbb{K}C\rightarrow\mathbb{K}S$ are defined (on basis elements)
by 
\[
\varphi(a)=\sum_{b\leq a}C(b),\quad\psi(x)=\sum_{y\leq x}\mu(y,x)S(y)
\]

where $\mu$ is the M�bius function of the poset $\leq$. 
\end{thm}

\begin{rem}
\thmref{MainThm} can be proved, \emph{mutatis mutandis}, for $E$-Ehresmann
and right restriction semigroups using $\leq_{l}$ instead of $\leq_{r}$.
\end{rem}

\begin{proof}[Proof of \thmref{MainThm}]
The proof that $\varphi$ and $\psi$ are bijections is identical
to that in \cite{Steinberg2006}, as is Case 1 of the proof that $\varphi$ is a homomorphism.  

For Case 2, we assume $\rb(C(a))\neq\db(C(b))$ (or equivalently, $a^{\ast}\neq b^{+})$.
Define $\tilde{a}=ab^{+}$ and $\tilde{b}=a^{\ast}b$. Proceeding as in \cite{Stein2016Ehresmann}, it remains to 
 show  that
\begin{equation}
\left(\sum_{a^{\prime}\leq\tilde{a}}C(a^{\prime})\right)\left(\sum_{b^{\prime}\leq\tilde{b}}C(b^{\prime})\right)=\left(\sum_{a^{\prime}\leq a}C(a^{\prime})\right)\left(\sum_{b^{\prime}\leq b}C(b^{\prime})\right).\label{eq:MultiplicationOfLinearCombinationWithTilde}
\end{equation}
We set  $x=C(a)$, $\tilde{x}=C(\tilde{a})$, $y=C(b)$ and
$\tilde{y}=C(\tilde{b})$ so \Eqref{MultiplicationOfLinearCombinationWithTilde}
can be written as 
\begin{equation}
\left(\sum_{x^{\prime}\leq\tilde{x}}x^{\prime}\right)\left(\sum_{y^{\prime}\leq\tilde{y}}y^{\prime}\right)=\left(\sum_{x^{\prime}\leq x}x^{\prime}\right)\left(\sum_{y^{\prime}\leq y}y^{\prime}\right).\label{eq:MultiplicationWithTildesInMorphisms}
\end{equation}
Now we need the new assumption of left restriction. First note that
$\tilde{y}\leq y$ and $\tilde{x}\leq_{l}x$ since they are restriction
and co-restriction respectively. $S$ is left restriction so $\tilde{x}\leq x$
by \lemref{LeqLInLeqRForRestriction}. Hence, every element on the
left hand side of \Eqref{MultiplicationWithTildesInMorphisms} appears
also on the right hand side. What is left to show is that a multiplication
$x^{\prime}\cdot y^{\prime}$ on the right hand side of \Eqref{MultiplicationWithTildesInMorphisms}
equals $0$ unless $x^{\prime}\leq\tilde{x}$ and $y^{\prime}\leq\tilde{y}$.
Take $x^{\prime}\leq x$ such that $x^{\prime}\nleq\tilde{x}$ and
assume that there is a $y^{\prime}\leq y$ such that $\exists x^{\prime}\cdot y^{\prime}$,
that is, $\rb(x^{\prime})=\db(y^{\prime})$. Since $y^{\prime}\leq y$
we have $\rb(x^{\prime})=\db(y^{\prime})\leq\db(y)$ by \cite[CO1]{Stein2016Ehresmann}.
Now, by \cite[EC7]{Stein2016Ehresmann} (choosing $f=\db(y)$) we
have that 
\[
(x^{\prime}\mid\rb(x^{\prime})\wedge\db(y))\leq(x\mid\rb(x)\wedge\db(y))
\]
but note that $(x\mid\rb(x)\wedge\db(y))=\tilde{x}$ and $\rb(x^{\prime})\wedge\db(y)=\rb(x^{\prime})$
so we get 
\[
x^{\prime}=(x^{\prime}\mid\rb(x^{\prime}))\leq\tilde{x}
\]
a contradiction. Similarly, take $y^{\prime}\leq y$ such that $y^{\prime}\nleq\tilde{y}$
and assume that there is an $x^{\prime}\leq x$ such that $\exists x^{\prime}\cdot y^{\prime}$,
that is, $\rb(x^{\prime})=\db(y^{\prime})$. Again, since $\rb(x^{\prime})\leq\rb(x)$
we have that $\db(y^{\prime})\leq\rb(x)$ and clearly $\db(y^{\prime})\leq\db(y)$
hence $\db(y^{\prime})\leq\rb(x)\wedge\db(y)=\db(\tilde{y})$. By
\cite[EC2]{Stein2016Ehresmann} there exists a restriction $(\db(y^{\prime})\mid\tilde{y})$.
But $(\db(y^{\prime})\mid\tilde{y})\leq\tilde{y}\leq y$ so by the
uniqueness of restriction $(\db(y^{\prime})\mid\tilde{y})=y^{\prime}$,
hence $y^{\prime}\leq\tilde{y}$, a contradiction. This finishes the
proof.
\end{proof}

\section{Consequences for the rest of \cite{Stein2016Ehresmann}}

\cite[Corollary 3.6]{Stein2016Ehresmann} should be reformulated as:

\begin{cor}
Let $S$ be an $E$-Ehresmann and left (or right) restriction semigroup
such that $E$ is finite, then $\mathbb{K}S$ is a unital algebra.
\end{cor}

\cite[Example 4.4]{Stein2016Ehresmann} should be deleted because
our weakened theorem no longer applies. The rest of the examples
in \cite[Section 4]{Stein2016Ehresmann} remain valid. In particular,
note that the monoid $\PT_{n}$ of \cite[Example 4.3]{Stein2016Ehresmann}
is left restriction and that \cite[2.5.11-2.5.12]{Cornock2011} proves
that every strong semilattice of monoids (\cite[Example 4.5]{Stein2016Ehresmann})
is a restriction semigroup.

The results of \cite[Subsection 5.1]{Stein2016Ehresmann} remain valid,
but one can draw conclusions regarding the algebras of the Ehresmann
semigroups only if they are left (or right) restriction. For instance,
it can be done for $(2,1,1)$ subalgebras of $\PT_{n}$ (\cite[Example 5.10]{Stein2016Ehresmann})
and finite ample semigroups (part of \cite[Example 5.11]{Stein2016Ehresmann}).

Likewise, the results of \cite[Subsection 5.2]{Stein2016Ehresmann}
are now valid only under the additional assumption of left or right
restriction. \cite[Proposition 5.17]{Stein2016Ehresmann} should be
reformulated as:

\begin{prop}
Let $S$ be a finite $E$-Ehresmann and left (or right) restriction
semigroup whose corresponding Ehresmann category $C$ is an EI-category.
Then $\mathbb{K}\Reg_{E}(S)$ is isomorphic to $\mathbb{K}S/\Rad(\mathbb{K}S)$. 
\end{prop}
The proof is as given in \cite{Stein2016Ehresmann}.

\bibliographystyle{plain}
\bibliography{library}